\documentclass{amsart}

\usepackage{amsmath}
\usepackage{amssymb}
\usepackage{amsthm}
\usepackage{tikz}
\usetikzlibrary{matrix,arrows,backgrounds,shapes.misc,shapes.geometric,patterns,calc,positioning}

\usepackage{ifthen}

\newcommand{\ZZ}{\mathbb{Z}}

\newcommand{\CC}{\mathbb{C}}
\newcommand{\NN}{\mathbb{N}} 
  
\newcommand{\KK}{\mathbb{K}}

\newcommand{\PP}{\mathbb{P}}

\newcommand{\col}{\mathrm{col}}

\newcommand{\imag}{\operatorname{Im}}

\newcommand{\Aut}{\operatorname{Aut}}  
\newcommand{\End}{\operatorname{End}}
\newcommand{\id}{\operatorname{id}} 
   
\newcommand{\Sp}{\operatorname{span}}

\newcommand{\Hom}{\operatorname{Hom}}

\newcommand{\Gr}{\operatorname{Gr}}

\newcommand{\GL}{\operatorname{GL}} 
\newcommand{\rep}{\operatorname{rep}}

\newcommand{\Res}{\operatorname{Res}}

\newcommand{\AlgGr}{\operatorname{AlgGr}} 
 
\newcommand{\msa}{\mathfrak{msa}} 
\newcommand{\supp}{\operatorname{supp}}

\newcommand{\Iso}{\operatorname{Iso}} 
\newcommand{\Inn}{\operatorname{Inn}}

\newtheorem{theorem}{Theorem}[section]
\newtheorem{proposition}[theorem]{Proposition}
\newtheorem{lemma}[theorem]{Lemma}
\newtheorem{corollary}[theorem]{Corollary}
\newtheorem{definition}[theorem]{Definition}

\newtheorem{example}[theorem]{Example}
\newtheorem{problem}[theorem]{Problem}

\theoremstyle{remark}

\renewenvironment{proof}{{\noindent\bf Proof.}}{\hfill $\Box$\par\vskip3mm}

\begin{document}    

\title[Automorphism Groups Acting on Subalgebra Varieties]{Automorphism Groups of Finite-Dimensional Algebras Acting on Subalgebra Varieties} 
\author{Alex Sistko} 
\address{Department of Mathematics, University of Iowa, 
Iowa City, IA 52242}
\email{alexander-sistko@uiowa.edu}

\maketitle{}
\date{}

\begin{abstract}

 Let $k$ be an algebraically-closed field, and $B$ a unital, associative $k$-algebra with $n := \dim_kB < \infty$. For each $1 \le m \le n$, the collection of all $m$-dimensional subalgebras of $B$ carries the structure of a projective variety, which we call $\AlgGr_m(B)$. The group $\Aut_k(B)$ of all $k$-algebra automorphisms of $B$ acts regularly on $\AlgGr_m(B)$. In this paper, we study the problem of explicitly describing $\AlgGr_m(B)$, and classifying its $\Aut_k(B)$-orbits. Inspired by recent results from \cite{IS}, we compute the homogeneous vanishing ideal of $\AlgGr_{n-1}(B)$ when $B$ is basic, and explictly describe its irreducible components. We show that in this case, $\AlgGr_{n-1}(B)$ is a finite union of $\Aut_k(B)$-orbits if $B$ is monomial or its Ext quiver is Schur, but construct a class of examples to show that these conditions are not necessary.
\end{abstract}  


\section{Introduction}\label{s.intro} 

Within contemporary algebra, finite-dimensional associative algebras undoubtedly play a critical role. For instance, their (finite-dimensional) representation theory can be reasonably said to unite the (finite-dimensional) representation theories of groups, quivers, and other categories of continuing interest. But they can also be studied directly, ex. as a subcategory of the category of associative rings, or algebras over a fixed base field. We might call this direct approach ``ring-theoretic,'' although this title perhaps obscures the difficulty and importance of the problems that arise, since noncommutative ring theorists do not seem particularly interested in finite-dimensional objects. For instance, this approach is necessary to understand automorphism groups of finite-dimensional algebras, which can be related to important symmetries of their module categories \cite{Bass}, \cite{Fr}. It is also necessary for the study of varieties of algebras of a fixed dimension, or for related classification problems \cite{deGraaf}, \cite{Gab}.  

To remedy this hopelessness, we can take a relative point of view. Instead of trying to classify $n$-dimensional algebras over a field $k$, we can start with such an algebra $B$, and show how conditions on $B$ influence its quotient algebras and subalgebras. Of the two, quotients appear to be better-behaved than subalgebras. Nevertheless, subalgebras of finite-dimensional algebras still remain important to contemporary mathematics: consider the theory of split and separable extensions \cite{Raf}, or the work of Motzkin and Taussky \cite{Motz1} \cite{Motz2} on commutative subalgebras of matrix rings. In particular, maximal subalgebras of finite-dimensional algebras, possibly with respect to a certain property (ex. commutativity), have been a subject of interest for decades. Maximal commutative subalgebras of $M_n(k)$ have been studied extensively, for instance in \cite{Sch}, \cite{Jac}, \cite{Mir}, \cite{Motz1}, \cite{Motz2}, \cite{Ger2}, \cite{Laff}. Maximal subalgebras of central simple algebras were classified by Racine in \cite{Rac0}, \cite{Rac}. Independently, Agore \cite{Ag} used a geometric argument of Gerstenhaber \cite{Ger} to bound the maximal dimension of a $\CC$-subalgebra in $M_n(\CC)$ by $n^2-n+1$. More recently, the author and Iovanov \cite{IS} found general structure theorems for maximal subalgebras of finite-dimensional algebras, provided explicit classifications for semisimple algebras and basic algebras, and connected their theory to split and separable extensions of algebras.

This paper was inspired by a question raised during the writing of \cite{IS}: if $B$ is a finite-dimensional algebra, and $A, A' \subset B$ are two isomorphic subalgebras, under what conditions do we know that $A' = \psi (A)$ for some $k$-algebra automorphism $\psi \in \Aut_k(B)$? More generally, can we classify $\Aut_k(B)$-orbits of subalgebras of $B$, and relate them to isoclasses of subalgebras? To study this question, it seems natural to try and exploit the structure of $\Aut_k(B)$, which is a linear algebraic group. For each $m \le \dim_kB$, we define a (projective) $k$-variety $\AlgGr_m(B)$ which parametrizes $m$-dimensional subalgebras, and on which $\Aut_k(B)$ acts regularly. We then seek to understand geometric properties of $\AlgGr_m(B)$, and properties of the $\Aut_k(B)$-action upon it. 

The idea of studying subalgebras through the variety $\AlgGr_m(B)$ was partially inspired from recent work in representation theory. Taking $B = kQ/I$ for some quiver $Q$ and admissible ideal $I$, representation theorists have long studied affine and projective varieties parametrizing $B$-modules of a fixed dimension vector $\underline{d} = (d_1,\ldots , d_n)$. For exposition, see \cite{Bong}, \cite{Geiss}, \cite{HZ}. The affine theory appears to extend at least as far back as Gabriel's theorem on representation-finite hereditary algebras \cite{HZ}, while the projective theory was developed by Bongartz and Huisgen-Zimmermann in \cite{BHZ1}, \cite{BHZ2}. A common unifying principle is the construction of moduli spaces in the spirit of King \cite{King}. On the algebra side, Gabriel \cite{Gab} constructs an affine variety $\operatorname{Alg}(n)$ which parametrizes $n$-dimensional algebras over a fixed field $k$. In a broad sense, one might hope that $\AlgGr_m(B)$ provides a projective context to study questions related to Gabriel's $\operatorname{Alg}(n)$, in analogy to the work described above. However, $\AlgGr_m(B)$ is in some sense fundamentally different, with its focus on subalgebras of a fixed algebra $B$, rather than all algebras of a fixed dimension. But again, this relative perspective has an analogy in geometric representation theory, where others have used projective varieties to study submodules of a fixed dimension vector. Indeed, if $M$ is a $\underline{d}$-dimensional $kQ$-module and $\underline{e} = (e_1,\ldots , e_n)$ is a dimension vector satisfying $e_i \le d_i$ for all $1\le i \le n$, the Quiver Grassmannian $\Gr_{\underline{e}}^Q(M)$ of all $\underline{e}$-dimensional $kQ$-submodules of $M$ has provoked recent interest. These varieties first appeared in the work of Schofield \cite{Scho} and Crawley-Boevey \cite{CB}. Subsequent work can be seen from Reineke \cite{Rein}, Caldero-Reineke \cite{CR}, and Cerulli Irelli-Feigin-Reineke \cite{CIFR1}, \cite{CIFR2}, \cite{CIFR3}. 

The paper is organized as follows. Section \ref{s.bg} supplies the main definitions used in this paper, and briefly reviews relevant results from \cite{IS}. In section \ref{s.sav} we introduce the variety $\AlgGr_m(B)$ of $m$-dimensional subalgebras of $B$, and discusses its basic properties. Section \ref{s.msa} contains the main results of this paper. In it, we give an explicit realization for $\AlgGr_{\dim_kB-1}(B)$ where $B$ is a basic, including a description of its irreducible components. More specifically, we prove the following: 

\begin{theorem} 
Let $B = kQ/I$, for some quiver $Q$ and admissible ideal $I$. Let $\dim_kB = n$. Then $\AlgGr_{n-1}(B)$ only depends on the Ext-quiver of $B$. Furthermore, there exists a closed embedding $\AlgGr_{n-1}(B) \hookrightarrow \mathbb{P}^{N}$, for some natural number $N$, under which the irreducible components of $\AlgGr_{n-1}(B)$ become linear subspaces. The homogeneous vanishing ideal of $\AlgGr_{n-1}(B)$ with respect to this embedding is explicitly computed, and its irreducible components are explicitly identified. 
\end{theorem} 

\noindent See theorem \ref{t.msa(Q)} for a description of the vanishing ideal and corollary \ref{c.irred} for the description of the irreducible components. As a consequence, we have the following:  

\begin{corollary}
If $B$ is basic, then $\AlgGr_{n-1}(B)$ is irreducible if and only if its Ext-quiver is either an $m$-loop quiver, and $m$-Kronecker quiver, or two isolated vertices.
\end{corollary}

 \noindent See definition \ref{d.irredQ} for all relevant terminology. 

 Finally, section \ref{s.iso} discusses the action of $\Aut_k(B)$ on $\AlgGr_{n-1}(B)$ in the basic case. In particular, we study conditions under which $\AlgGr_{n-1}(B)$ is a finite union of $\Aut_k(B)$-orbits. Using results from \cite{IS}, we prove the following theorem: 
 
 \begin{theorem} 
 Let $B$, $Q$, and $n$ be as before. Let $H_B$ be the group of all $B$-automorphisms fixing the vertices of $Q$. Then $\AlgGr_{n-1}(B)$ is a finite union of $\Aut_k(B)$-orbits if and only if for every pair $(u,v) \in Q_0\times Q_0$ with $uJ(B)/J(B)^2v \neq \{ 0\}$, $\mathbb{P}((uJ(B)/J(B)^2v)^*)$ is a finite union of $H_B$-orbits.
 \end{theorem} 
 
\noindent See Theorem \ref{t.orb} for details. In the above, $(uJ(B)/J(B)^2v)^*$ denotes the vector space dual. As a consequence, we show that $\AlgGr_{n-1}(B)$ is a finite union of $\Aut_k(B)$-orbits in the case that either $Q$ is Schur, or $B$ is monomial. We end this section by constructing a class of $B$ satisfying this finite orbit property, but for which $B$ is not generally monomial and its Ext-quiver is not generally Schur.


\section{Background}\label{s.bg}

Unless otherwise stated, $k$ will denote an algebraically-closed field. All algebras are unital, associative, finite-dimensional $k$-algebras. Our terminology on bound quiver algebras essentially comes from \cite{ASS}. Let $Q$ be a finite quiver with vertex set $Q_0$, arrow set $Q_1$, and source (resp. target) function $s$ (resp. $t$) $: Q_1\rightarrow Q_0$. Let $kQ$ denote the path algebra of $Q$, and let $J(Q)$ denote the two-sided ideal in $kQ$ generated by $Q_1$. By a slight abuse of notation, for any $u, v \in Q_0$ we let $uQ_1v$ denote the set of arrows in $Q$ with source $u$ and target $v$, and we let $ukQ_1v$ denote their $k$-span inside $kQ$. Define $d(u,v) = \dim_kukQ_1v$. Note that if $uQ_1v = \emptyset$, then $ukQ_1v = \{ 0 \}$ and $\GL(ukQ_1v)$ is the trivial group. Similar to \cite{GS2}, we define $V^2(Q) = \{ (u,v) \in Q_0\times Q_0\mid uQ_1v \neq \emptyset\}$. The \emph{no-double edge graph} of a quiver $Q$ is the (undirected) graph with vertex set $Q_0$, and an edge between $u$ and $v$ if and only if $uQ_1v \cup vQ_1u \neq \emptyset$.  

A \emph{basic algebra} is an algebra of the form $B = kQ/I$, where $I$ is an \emph{admissible ideal} of $kQ$, i.e. an ideal satisfying $J(Q)^2 \supset I \supset J(Q)^{\ell}$ for some $\ell \geq 2$. Note that $B = kQ_0 \oplus J(B) = kQ_0 \oplus J(Q)/I$, and that $kQ_0 \cong B/J(B) \cong k^{|Q_0|}$. In fact, the Wedderburn-Malcev theorem tells us that this decomposition is unique in the following sense: for all subalgebras $B_0 \subset B$ isomorphic to $kQ_0$, there is an $x \in J(B)$ such that $(1+x)B_0(1+x)^{-1} = kQ_0$ (see, for instance, \cite{Pi}).

 For $n \geq 2$, we define $T_nQ := kQ/J(Q)^n$, the \emph{$n^{th}$ truncated quiver algebra} associated to $Q$. From here on out, we put total orderings on $Q_0$ and $Q_1$: write $Q_0 = \{ v_0 < \ldots < v_{n_0}\}$ for some $n_0 \in \ZZ_{\geq 0}$ and $Q_1 = \{ \alpha_1<\ldots < \alpha_{n_1} \}$ for some $n_1 \in \NN$. Note that for the sake of stating the results with a minimum of caveats, the author assumes $Q_1 \neq \emptyset$ for our generic quiver $Q$; the theorems are still valid, albeit trivial, for the $Q_1 = \emptyset$ case. We also assume that the total ordering on $Q_1$ satisfies the following condition: if $uQ_1v \neq \emptyset$, then the arrows in $uQ_1v$ form an interval within the total ordering of $Q_1$ (i.e. there exist indices $1 \le i_{uv}  \le n_1$ and $n_{uv} \geq 0$ such that $ukQ_1v = \{ \alpha_{i_{uv}}, \alpha_{i_{uv}+1},\ldots , \alpha_{i_{uv}+n_{uv}} \}$). 

We let $\Aut_k(B)$ denote the group of all $k$-algebra automorphisms of $B$. It is a Zariski-closed subgroup of $\GL(B)$, and hence a linear affine algebraic group. Our notation for subgroups of $\Aut_k(B)$ is borrowed from the notation in \cite{Poll}, \cite{GS1}, \cite{GS2}. Given two subalgebras $A$ and $A'$ of $B$ and a subgroup $G \le \Aut_k(B)$, we say they are \emph{$G$-conjugate} in $B$ if there exists a $\phi \in G$ such that $\phi (A) = A'$. If $G = \Aut_k(B)$, we just say that they are conjugate. For a unit $u \in B^{\times}$, we let $\iota_u$ denote the corresponding \emph{inner automorphism}, i.e. the map $\iota_u(x) = uxu^{-1}$ for all $x \in B$. We let $\Inn (B)$ denote the group of all inner automorphisms, and $\Inn^*(B) = \{ \iota_{1+x} \mid x \in J(B)  \}$ denote the group of \emph{unipotent inner automorphisms}. For a quiver $Q$, we define $S_Q$ to be the permutations $\sigma$ of $Q_0$ satisfying $|\sigma (u)Q_1\sigma (v)| = |uQ_1v|$ for all $u,v \in Q_0$. This is a quotient of $\Aut(Q)$, and with our ordering of $Q_1$, it acts as a subgroup of $\Aut_k(T_nQ)$ by defining it on arrows as follows: for $\sigma \in S_Q$, and $(u,v) \in V^2(Q)$ with $uQ_1v = \{ \alpha_i,\ldots \alpha_{i+d} \}$ and $\sigma (u) Q_1\sigma (v) = \{ \alpha_j , \ldots , \alpha_{j+d}\}$, $\sigma (\alpha_{i+\ell}) = \alpha_{j+\ell}$.    

Most of the constructions outlined in the next section work for arbitrary $B$, but basic algebras will be the primary focus of this paper. There is a simple reason for this: subalgebras of basic algebras are particularly well-behaved.  For instance, maximal subalgebras of basic algebras can be classified explicitly up to $\Inn^*(B)$-conjugation. Indeed, theorem 4.1 of \cite{IS} shows the following: 

\begin{theorem}\label{t.basic}
Let $B = kQ/I$ be a basic algebra over an algebraically-closed field $k$. Let $A \subset B$ be a maximal subalgebra. Consider the following two classes of maximal subalgebras of $B$: 
\item For a two-element subset $\{ u,v \} \subset Q_0$, we define  
\begin{equation*}
A(u+v) := k(u+v)\oplus \left( \bigoplus_{w \in Q_0\setminus\{ u,v\}}{kw} \right) \oplus J(B). 
\end{equation*}
\item For an element $(u,v) \in V^2(Q)$ and a codimension-$1$ subspace $U \le ukQ_1v$, we define 
\begin{equation*} 
A(u,v,U) := kQ_0 \oplus U \oplus \left( \bigoplus_{(w,y) \in Q_0^2\setminus\{(u,v)\}}{wkQ_1y} \right) \oplus J(B)^2. 
\end{equation*}  

\noindent Then there exists a unipotent inner automorphism $\iota_{1+x} \in \Inn^*(B)$ such that either $\iota_{1+x}(A) = A(u+v)$ or $\iota_{1+x}(A) = A(u,v,U)$, for some appropriate choice of $u$, $v$, and possibly $U$.
\end{theorem}  

\noindent Following the terminology in \cite{IS}, if $A$ is $\Inn^*(B)$-conjugate to a subalgebra of the form $A(u+v)$, then we say that $A$ is of \emph{separable type}. If $A$ is $\Inn^*(B)$-conjugate to an algebra of the form $A(u,v,U)$, then we say that $A$ is of \emph{split type}. Theorem \ref{t.basic} has immediate consequences for subalgebras of basic algebras. These properties are not difficult to demonstrate, but since they are conceptually useful, we list them in a corollary below for future reference: 

 \begin{corollary}\label{c.inherit}
 Let $B$ be a basic $k$-algebra of dimension $n$, and let $A \subset B$ be a subalgebra. Then $A$ satisfies the following: 
\begin{enumerate} 
\item $A$ is also a basic algebra.
\item If $A$ is a maximal subalgebra, then $\dim_kA = n-1$.
\item If $A$ is a maximal subalgebra, then $J(A)$ is a $B$-subbimodule of $J(B)$, $J(A) = A \cap J(B)$, and $J(B)^2 \subset J(A)$.
\item More generally, if $m = \dim_kA$, then $J(B)^{2(n-m)} \subset A$. 
\end{enumerate} 
\end{corollary} 

\begin{proof} 
Properties 1-3 follow immediately from Theorem \ref{t.basic}. Property 4 follows from 1-3 by induction on $n-m$.
\end{proof} 

\noindent For arbitrary $B$, properties 1 and 2 are false. Property 3 only holds for a suitable generalization of split-type subalgebras, although it follows from \cite{IS} that, in general, $J(B)^2 \subset J(A)$. Similar reasoning also tells us that if we have a sequence $A_m \subset A_{m-1} \subset \ldots \subset A_1 \subset A_0 = B$, such that $A_i$ is a maximal subalgebra of $A_{i-1}$ for all $i$, then $J(B)^{2m} \subset A_m$. \newline

\noindent {\bf{Note:}}  In spite of this paper's focus on basic algebras, there are good reasons to study $\AlgGr_m(B)$ when $B$ is not basic. In the next section, for instance, we will see why subalgebras of $M_n(k)$ might be of particular interest. \newline


\section{Varieties of Subalgebras}\label{s.sav}

Set $n = \dim_kB$, and fix some  $1\le m \le n$. For any two finite-dimensional $k$-vector spaces $V$ and $W$, we let $\Hom_k^{\circ}(V,W)$ denote the collection of injective $k$-linear maps $V \rightarrow W$. We think of $\Hom_k^{\circ}(V,W)$ as $(\dim_kW)\times (\dim_kV)$-matrices of full rank as needed. $\GL(V)$ acts on the right of $\Hom^{\circ}(V,W)$, and the projection $\Hom_k^{\circ}(V,W) \rightarrow \Hom_k^{\circ}(V,W)/\GL(V)$ is a geometric quotient in the category of schemes. Of course, $\Hom_k^{\circ}(V,W)/\GL(V)$ is just $ \Gr_{\dim_kV}(W)$, the Grassmannian of $\dim_kV$-dimensional subspaces of $W$. 

\begin{definition} Let $\AlgGr_m(B)$ denote the collection of all $m$-dimensional subalgebras of $B$, considered as a subset of $\Gr_m(B)$. Note that if $B$ is basic, then $\AlgGr_{n-1}(B)$ is the collection of all maximal subalgebras of $B$.
\end{definition}  

\begin{lemma}\label{AlgGrclosed}
$\AlgGr_m(B)$ is closed in $\Gr_m(B)$. In particular, $\AlgGr_m(B)$ is a projective variety.
\end{lemma} 

\begin{proof} 
Choose a basis of $B$ to identify $U:= \Hom_k^{\circ}(k^m,B)$ with the open subset of $M_{n,m}(k)$ consisting of matrices with full rank. Let $\pi : U \rightarrow \Gr_m(B)$ denote the corresponding quotient map. Since $\pi$ is a geometric quotient, it is enough to show that $ X:= \pi^{-1}(\AlgGr_m(B))$ is closed in $U$. Now, $X$ is the collection of all matrices $A \in U$ satisfying the following: 
\begin{enumerate} 
\item $1 \in \col A$.
\item $\imag (\mu_B\mid_{\col A\otimes \col A} ) \subset \col A$.
\end{enumerate} 
Here, $\col A$ denotes the column space of $A$, and $\mu_B : B\otimes B \rightarrow B$ is the multiplication map of $B$ (considered as a map $k^{n^2} \rightarrow k^n$ when needed). It is standard to check that these conditions are Zariski-closed, from which the claim follows. 
\end{proof}     

 At this point, we wish to take a quick look at $\AlgGr_m(B)$ for $B = M_d(k)$, where $d>1$. This example will show us two things: first, it will show us that these varieties encode interesting algebraic data; and second, it will give us a sense of how difficult it might be to provide a uniform description of $\AlgGr_m(B)$ for arbitrary $m$ and $B$.   
 
 \begin{definition} 
For any $A \in \AlgGr_m(B)$, define $\Iso (A,B) = \{ A' \in \AlgGr_m(B) \mid A' \cong A$ as $k$-algebras$\}$. 
\end{definition}

 \begin{proposition}\label{t.matrix}
 Let $d>1$, and let $1 \le m \le d^2$. Then points of $\AlgGr_m(M_d(k))$ may be identified with $d$-dimensional faithful representations of 
 $m$-dimensional algebras, in such a way that the following hold, for each $A \in \AlgGr_m(M_d(k))$: 
 \begin{enumerate} 
 \item $\Iso (A,M_d(k))$ is the collection of all faithful $d$-dimensional $A$-modules;
 \item $A' \in \Iso (A,M_d(k))$ is conjugate to $A$ if and only if they correspond to isomorphic $A$-modules.
 \end{enumerate}
 \end{proposition}  
 
 \begin{proof} 
 Let $A$ be an $m$-dimensional $k$-algebra. Then a $d$-dimensional faithful module is the same as a $d$-dimensional vector space $V$, along with a $k$-algebra injection $A \hookrightarrow \End_k(V)$. After fixing a basis for $V$, this is the same as a $k$-algebra injection $A \hookrightarrow M_d(k)$. A different faithful $d$-dimensional representation, call it $V'$, will yield another injection $A \hookrightarrow \End_k(V') \cong M_d(k)$, whose image is isomorphic to $A$. Every subalgebra of $M_d(k)$ isomorphic to $A$ arises in the way, and so elements of $\AlgGr_m(M_d(k))$ correspond bijectively to $d$-dimensional faithful representations of $m$-dimensional algebras. For claim (1), recall that $\Aut_k(M_d(k)) = \Inn(M_d(k)) \cong \GL_d(k)/k^{\times}$ by the Skolem-Noether Theorem. If $g \in \GL_d(k)$, then restriction along $\phi:= \iota_g\mid_A : A \rightarrow gAg^{-1}$ replaces the faithful $A$-module $k^d$ by the faithful $A$-module $\Res_{\phi}(k^d)$. But $\Res_{\phi}(k^d) \cong k^d$ as $A$-modules under $g$, considered as an endomorphism of $k^d$. Conversely, suppose that $A' \in \AlgGr_m(M_d(k))$ is given, such that $A \cong A'$ and the $A$-module structure on $k^d$ induced by $A'$ is isomorphic to the structure defining $A$ itself. This means that there exists a $k$-algebra isomorphism $\psi : A \rightarrow A'$ such that $\Res_{\psi}(k^d) \cong k^d$ as $A$-modules. In other words, there is an invertible $k$-linear map $g \in \GL_d(k)$ such that for all $a \in A$ and $v \in k^d$, $g(\psi(a)v) = ag(v)$. But since $k^d$ is faithful, this implies $g\psi(a) = ag$ as endomorphisms on $k^d$, and hence $\psi = \iota_{g^{-1}}\mid_A$. 
 \end{proof} 
 
 \noindent {\bf{Note:}} Theorem \ref{t.matrix} implies that $\AlgGr_m(B)$ is not always a finite union of $\Aut_k(B)$-orbits. For instance, consider the three-dimensional algebra $A = k[x,y]/(x^2,xy,y^2)$. $A$ has infinite representation type. In fact, it is easy to see that for each $\lambda \in k$, the map $(x,y) \mapsto \left(\left(\begin{array}{cc} 0 & 1 \\ 0 & 0 \end{array} \right), \left(\begin{array}{cc} 0 & \lambda \\ 0 & 0 \end{array} \right) \right)$ defines a $2$-dimensional indecomposable $A$-module $M_{\lambda}$, and $M_{\lambda} \cong M_{\mu}$ if and only if $\lambda = \mu$. But then $\{ M_{\lambda} \oplus A\}\mid_{\lambda \in k}$ is an infinite collection of pairwise non-isomorphic, $5$-dimensional faithful $A$-modules. In particular, $\AlgGr_3(M_5(k))$ is not a finite union of $\Aut_k(M_5(k))$-orbits. It may amuse the reader to note that $\AlgGr_3(M_2(k))$ \emph{is} a finite union of $\Aut_k(M_2(k))$-orbits.   \newline  
 
 Isomorphism classes of subalgebras of $B$ are generally not open or closed, but they are constructible in $\AlgGr_m(B)$. This may not surprise the reader, since they are clearly $\Aut_k(B)$-invariant, and hence a union of $\Aut_k(B)$-orbits. Note, however, that the example above shows that this fact alone does not imply constructibility, since isoclasses might be (uncountably) infinite unions of $\Aut_k(B)$-orbits. To prove the constructibility of isoclasses, we rely on a well-known theorem of Chevalley: 

\begin{theorem}[Chevalley]\label{t.chevalley} 
Let $f : X \rightarrow Y$ be a morphism of schemes. Assume that $f$ is quasi-compact and locally of finite presentation, and that $Y$ is quasi-compact and quasi-separated. Then the image of every constructible subset of $X$ under $f$ is a constructible subset of $Y$.
\end{theorem}

\noindent It is easy to check that $\pi: U \rightarrow \Gr_m(B)$ of lemma \ref{AlgGrclosed} satisfies these conditions. To prove the constructibility of isoclasses, we introduce a few extra definitions. 


\begin{definition} 
Let $\pi : U \rightarrow \Gr_m(B)$ be the quotient map from before. For any $A \in \AlgGr_m(B)$, an element $F \in \pi^{-1}(A)$ is called a \emph{frame} of $A$.
\end{definition} 

\begin{definition} 
Let $F$ be a fixed frame for $A \in \AlgGr_m(B)$, with columns $f_1,\ldots , f_m$. Define $I(F)$ to be the group of all $g \in \GL(B)$ such that $g\cdot\mu_B(f_i,f_j) = \mu_B(g\cdot f_i, g\cdot f_j)$ for all $i,j \in \{ 1,\ldots , m\}$. 
\end{definition} 

\noindent  Of course, $I(F)$ 
is a closed subgroup of $\GL(B)$, and $I(F)$ contains $\Aut_k(B)$. Note that $I(F)$ depends \emph{a priori} on the frame chosen for $A$. 


\begin{proposition} 
Fix $A \in \AlgGr_m(B)$ and $F \in \pi^{-1}(A)$ with columns $f_1,\ldots , f_m$. Then $A' \in \Iso (A,B)$ if and only if for any frame $F'$ of $A'$, there exist $g \in I(F)$ and $\phi \in \GL_m(k)$ such that $F' = g\cdot F \cdot \phi^{-1}$. Consequently, $\Iso (A,B)$ is a constructible subset of $\AlgGr_m(B)$.
\end{proposition}  

\begin{proof} 
It is clear that $A' \cong A$ if any frame of $A'$ is of the form $g\cdot F\cdot \phi^{-1}$, for some $g \in I(F)$ and $\phi \in \GL_m(k)$. So assume that $\gamma : A \rightarrow A'$ is a $k$-algebra isomorphism. Then $\gamma(F) := (\gamma (f_1) \mid \cdots \mid \gamma (f_m) )$ is a frame for $A'$. Since any two frames for $A'$ are in the same $\GL_m(k)$-orbit of $U$, it is enough to show that $\gamma(F) = g\cdot F$ for some $g \in I(F)$. Pick ordered bases $\beta = \{ f_1,\ldots , f_m , f_{m+1},\ldots , f_n\}$ and $\gamma (\beta) := \{ \gamma (f_1),\ldots , \gamma (f_m), h_{m+1}, \ldots , h_n\}$ for $B$, and define $\hat{\gamma}(f_i) = \gamma (f_i)$ for $1 \le i \le m$ and $\hat{\gamma}(f_{m+j} ) = h_{m+j}$ for $1 \le j \le n-m$. Then if $g$ denotes the matrix corresponding to $\hat{\gamma}$ with respect to the fixed basis chosen for $B$, $g \in I(F)$ and $\gamma (F) = g\cdot F$. From this, it follows that $\pi^{-1}(\Iso (A,B)) = \left[I(F)\times \GL_m(k)\right]\cdot F$ is constructible. By theorem \ref{t.chevalley}, $\Iso (A,B)$ is constructible as well.
\end{proof}   



\noindent We now begin our study of subalgebra varieties of basic algebras. We first note a straightforward consequence of corollary \ref{c.inherit}:

\begin{lemma}\label{t.red} 
Let $B = kQ/I$. Then $\AlgGr_m(B) \cong \AlgGr_{\overline{m}}(B/J(B)^{2(n-m)})$ as a projective variety, where $\overline{m} = \dim_kB/J(B)^{2(n-m)} -n+m$.
\end{lemma} 

\begin{proof} 
By corollary \ref{c.inherit}, $J(B)^{2(n-m)} \subset A$ for all $A \in \AlgGr_m(B)$. Therefore, the correspondence $A \leftrightarrow A/J(B)^{2(n-m)}$ provides an equivalence $\AlgGr_m(B) \cong \AlgGr_{\bar{m}}(B/J(B)^{2(n-m)})$, as we wished to show.
\end{proof} 

\noindent {\bf{Note:}} Applying lemma \ref{t.red} to maximal subalgebras of $B = kQ/I$, we find that $\AlgGr_{n-1}(B)$ is isomorphic to $ \AlgGr_{\dim_kB/J(B)^2-1}{B/J(B)^2}$. But note that the algebra $B/J(B)^2$ is isomorphic to $ T_2Q$, so that $\AlgGr_{n-1}(B)$ \emph{only depends on the underlying quiver of $B$}. Therefore, we make the following definition:

\begin{definition}\label{msaDef}
$\msa (Q):= \AlgGr_{n-1}(T_2Q)$ where $\dim_kT_2Q = n$. We can think of $\msa (Q)$ as the variety of maximal subalgebras of \emph{any} algebra of the form $B = kQ/I$, for some admissible ideal $I \subset kQ$. The action of $\Aut_k(B)$ on $\msa (Q)$ factors through the projection map $\Aut_k(B) \rightarrow \Aut_k(B/J(B)^2) \cong \Aut_k(T_2Q)$.
 \end{definition} 

 Lemma \ref{t.red} shows that $\AlgGr_m(B) \cong \AlgGr_{m'}(B')$ does not imply $m = m'$ or $B \not\cong B'$. This problem is particularly acute for maximal subalgebras of basic algebras, where the corresponding variety for \emph{any} admissible quotient of $kQ$ is isomorphic to $\msa (Q)$. However, one might try to recover the subalgebra structure of $B = kQ/I$ from $\msa (Q)$ by studying its $\Aut_k(B)$-orbits and isoclasses. Since $\Aut_k(B)$ acts on $\msa (Q)$ through its image under the map $\Aut_k(B) \rightarrow \Aut_k(B/J(B)^2)$, this is really a special case of the following general problem:


\begin{problem}  
Suppose that $B$ is a finite-dimensional algebra, and that $G$ is a subgroup of $\Aut_k(B)$. Classify the $G$-orbits of $\AlgGr_m(B)$. 
\end{problem} 

\noindent This problem is too general for us to expect any reasonable answer, even for basic algebras. However, in section \ref{s.iso} we discuss several interesting conditions on an admissible ideal $I \subset kQ$, which guarantee that $\msa (Q)$ is at least a finite union of $\Aut_k(B)$-orbits. 


\section{The Structure of $\msa (Q)$}\label{s.msa}  

In this section, we provide an explicit description for the variety $\msa (Q)$ introduced in the previous section. Specifically, we construct a closed embedding $\msa (Q) \hookrightarrow \mathbb{P}^N$ for $N = |Q_0| + |Q_1| -2$, and find a finite set of generators for the homogeneous vanishing ideal of $\msa (Q)$ in $k[\mathbb{P}^N]$. We then describe the irreducible components of $\msa (Q)$.

For the next theorem, let $B = T_2Q$ for some quiver $Q$, with $Q_0 = \{ v_0, v_1,\ldots , v_{n_0}\}$, and $Q_1 = \{ \alpha_1,\ldots , \alpha_{n_1}\}$. Let $Q_0$ and $Q_1$ be ordered as in Section \ref{s.bg}. Extend these orderings to a total ordering on $Q_0\cup Q_1$ by defining $v< \alpha$ for all $v \in Q_0$ and $\alpha \in Q_1$. We will freely treat subsets of $Q_0\cup Q_1$ as totally ordered sets under the induced ordering. Let $V = \Sp_k(Q_0\setminus\{ v_0 \} \cup Q_1)$, and let $R= k[V] = k[v_1,\ldots , v_{n_0},\alpha_1,\ldots , \alpha_{n_1}]$ denote the homogeneous coordinate ring of the isomorphic projective varieties: 
\begin{equation*}
\Gr_{n_0+n_1-1}(V) \cong \PP(V^*) \cong \PP (V).  
\end{equation*} 

 Recall that if we consider an element $x \in Q_0\setminus\{v_0\} \cup Q_1$ as an element of $R$, then the statement ``$x$ vanishes on $A \in \Gr_{n_0+n_1-1}(V)$'' is equivalent to the statement ``$x(F) = 0$ for any frame $F$ of $A$.'' Here, a frame of $A$ is simply an $(n_0+n_1)\times (n_0+n_1-1)$-matrix whose columns form an ordered basis for $A$, and $x(F)$ is the $(n_0+n_1-1)\times(n_0+n_1-1)$-minor of $F$ formed by deleting the $x^{th}$-row. Note that if $F'$ is obtained from $F$ by column operations, then $x(F) = 0$ implies $x(F') = 0$. Also recall that $\Gr_{n_0+n_1-1}(V)$ can be decomposed into Schubert cells indexed by $(n_0+n_1-1)$-element subsets of $Q_0\setminus\{v_0\} \cup Q_1$. For our purposes, it will be easier to index the Schubert cells by $1$-element subsets instead, i.e. by taking the complement of a given $(n_0+n_1-1)$-element subset. With this convention, if $x \in Q_0\setminus\{v_0\} \cup Q_1$ is any element, then its corresponding Schubert cell is the collection of all $A \in \Gr_{n_0+n_1-1}(V)$ with the following property: any frame $F$ of $A$ is column-equivalent to a reduced-column echelon matrix with pivots on every row except for the $x$-row.

The basis $Q_0 \cup Q_1$ for $B$ allows us to fix a $k$-vector space isomorphism $B \rightarrow B^*$. The dual vector of $x \in Q_0 \cup Q_1$ will be denoted $x^*$. More generally, for any vector $x = \sum_{v \in Q_0}{\lambda_vv} + \sum_{\alpha\in Q_1}{\lambda_{\alpha}\alpha}$ in $B$, we let $x^*$ denote its image $\sum_{v \in Q_0}{\lambda_vv^*}+\sum_{\alpha \in Q_1}{\lambda_{\alpha}\alpha^*}$ in $B^*$. Finally, if $S$ is any finite set, and $\sigma : S \rightarrow \ZZ_{\geq 0}$ is function taking values in the non-negative integers, we let $|\sigma | = \sum_{x \in S}{\sigma (x)}$ and $\supp (\sigma) = \{ x \in S \mid \sigma (x) \neq 0\}$.

\begin{theorem}\label{t.msa(Q)}
There is a closed embedding $\msa (Q) \hookrightarrow \Gr_{n_0+n_1-1}(V)$, whose image is the vanishing set of the ideal $X = X_0 + X_{1/2} + X_1 \subset R$, where:
\begin{enumerate} 
\item $X_0$ is the ideal generated by the set $\{ v_i^2v_j - (-1)^{j-i-1}v_iv_j^2 \mid 1 \le i < j \le n_0 \} \cup \{ v_iv_jv_k \mid 1 \le i < j < k \le n_0\}$,  
\item $X_{1/2}$ is the ideal generated by the sets $\{\alpha_iv_j \mid v_j \not\in \{ s(\alpha_i),t(\alpha_i)\} \}$, $\{ (v_k - (-1)^{k-j-1}v_j)\alpha_i \mid 1\le j < k \le n_0, \alpha_i \in v_jQ_1v_k \}$, and $\{ v\alpha \mid s(\alpha) = v = t(\alpha) \}$, 
\item $X_1$ is the ideal generated by the set $\{ \alpha_i\alpha_j \mid \alpha_i$ is not parallel to $\alpha_j \}$.
\end{enumerate}
\end{theorem} 

\begin{proof} 
Let $v_0^* : B \rightarrow k$ denote the algebra character afforded by the simple $B$-module at $v_0$, and let $L : B \rightarrow \ker v_0^* = V$ be the linear transformation $L(x) = x - v_0^*(x)\cdot 1$. Then $L$ is surjective with $\ker L = k\cdot 1$, and so it induces an isomorphism $\overline{L}: B/k\cdot 1 \rightarrow V$. Dualizing, we get an isomorphism $(\overline{L})^* : V^* \rightarrow (B/k\cdot 1)^*$, which induces an isomorphism $\PP(V^*) \cong \PP((B/k\cdot 1)^*) \cong \Gr_{n_0+n_1-1}(B/k\cdot 1)$. But since every subalgebra is unital, the map $A \mapsto A/k\cdot 1$ induces a closed immersion $\msa (Q) \rightarrow \Gr_{n_0+n_1-1}(B/k\cdot 1)$, and hence a closed immersion $\msa (Q) \rightarrow \PP(V^*)$. Note that $V\cdot V \subset V$ as a subset of $B$, and that the image of $\msa (Q)$ in $\PP(V^*) \cong \Gr_{n_0+n_1-1}(V)$ is simply the set of all multiplicatively-closed $(n_0+n_1-1)$-dimensional subspaces of $V$. By a slight abuse of notation, we will sometimes identify $A \in \msa (Q)$ with its image $L(A) \subset V$.

Under this identification, we first show that every element of $X$ vanishes on $\msa (Q)$. To show that $X_0$ vanishes on $\msa (Q)$, we first consider $v_i(A(v_j+v_k))$, where $1 \le i \le n_0$ and $0 \le j < k \le n_1$. If $j=0$, then $L(A(v_0+v_k))$ has $Q_0\setminus\{v_0,v_k\} \cup Q_1$ as an ordered basis; by performing column operations on any frame for $L(A(v_0+v_k))$, we can transform it into a frame with precisely these columns. But then the $v_k$-row in the frame is zero, so that if $i \neq k$, $v_i(A(v_0+v_k)) = 0$. This immediately implies that $v_iv_{\ell}v_{p}(A(v_0+v_k)) = 0$ for any triple $(i,\ell,p)$ of pairwise distinct positive integers. Similarly, for all $1 \le i < j \le n_0$, either $i \neq k$ or $j \neq k$, and so $[v_i^2v_j - (-1)^{j-i-1}v_iv_j^2](A(v_0+v_k)) = 0$. Otherwise $j>0$, and $A(v_j+v_k)$ has a basis of the form $Q_0\setminus \{v_0,v_j,v_k\} \cup \{ v_j+v_k\} \cup Q_1$. We can totally order this basis by defining $v_i< v_j+v_k$ for $i<j$ and $v_j+v_k < v_i$ for $i>j$. Then $v_i(A(v_j+v_k)) = 0$ for all $i \not\in \{ j,l\}$, since the $v_j$-row equals the $v_k$-row in this frame. Again, we find that $[v_iv_{\ell}v_p](A(v_j+v_k)) = 0$ for any triple $(i,\ell,p)$ of pairwise distinct positive integers, and that $[v_i^2v_p - (-1)^{p-i-1}v_iv_p^2](A(v_j+v_k)) = 0$ if $\{ i,p\} \neq \{ j, k \}$. A direct computation shows that if $j<k$, then $v_k(A(v_j+v_k)) = 1$ and $v_j(A(v_j+v_k)) = (-1)^{k-j-1}$ with the specified frame. Hence, $[v_j^2v_k - (-1)^{k-j-1}v_jv_k^2](A(v_j+v_k))=0$. In other words, the generators of $X_0$ vanish on maximal subalgebras of separable type.  

Now consider $v_i(A)$, for $1 \le i \le n_0$ and $A$ a maximal subalgebra of split type. $A$ is $\Inn^*(B)$-conjugate to an algebra of the form $A(v_j,v_k,W)$, for $v_j,v_k \in Q_0$. If $i \not\in \{ j,k\}$, then $v_i \in L(A)$, and so there is a frame for $A$ with $v_i$ as a column. The minor of this frame formed by deleting the $v_i$-row contains a zero column, and so $v_i(A) = 0$. Immediately, we see that $[v_iv_{\ell}v_p](A) = 0$ for any three distinct vertices, and $[v_i^2v_p - (-1)^{p-i-1}v_iv_p^2](A) = 0$ if $\{ i,p\} \neq \{ j,k\}$. So it only remains to check the case $j \neq k$ and $\{ i ,p\} = \{j,k \}$. Note that $A$ must belong to a Schubert cell corresponding to an arrow $v_j\xrightarrow[]{\alpha_p}v_k$. This means that $L(A)$ contains a basis of the form $\gamma = Q_0\setminus \{ v_0, v_j,v_k\} \cup \{ v_j+\lambda\alpha_p, v_k - \lambda\alpha_p\} \cup Q_1\setminus v_jQ_1v_k \cup \beta$, where $\beta$ is a basis for $W$. Note that if $v_j = v_k$, we necessarily have $\lambda = 0$. In fact, if $\{ \alpha_i, \ldots , \alpha_{p-1}, \alpha_p, \ldots , \alpha_q\}$ are the (ordered) arrows from $v_j$ to $v_k$, then $\beta$ can be taken to be of the form $\beta = \{ \alpha_i+\lambda_i\alpha_p, \ldots , \alpha_{p-1}+\lambda_{p-1}\alpha_p, \alpha_{p+1}, \ldots , \alpha_q\}$. Ordering this basis for $L(A)$ in the obvious way, we obtain a frame for $L(A)$ satisfying $v_j(A) = (-1)^{k-j-1}v_k(A)$. Hence, $[v_j^2v_k - (-1)^{k-j-1}v_jv_k^2](A) = 0$. This shows that elements of $X_0$ belong to the vanishing ideal of $\msa (Q)$. 

Observe that $\alpha_i(A(v_j+v_k)) = 0$ because $\alpha_i \in L(A(v_j+v_k))$. So we only need to check the remaining relations on maximal subalgebras of split type. To show that elements of $X_{1/2}$ vanish on $\msa (Q)$, suppose that $A$ is conjugate to $A(v_j,v_k,W)$ as before, and take $\gamma$ to be the ordered basis described above. If $v_p\not\in \{ v_j,v_k\}$ then we have already seen $[\alpha_iv_p](A)= 0$, so we may assume that $v_p \in \{v_j,v_k\}$. Then if $\alpha_iv_p \in X_{1/2}$, $v_p$ must satisfy $v_p \not\in \{ s(\alpha_i),t(\alpha_i)\}$. This implies that $\alpha_i$ is not parallel to $\alpha_p$. Then $\alpha_i \in L(A)$ and $\alpha_i(A) = 0$. We must also show that for all arrows $\alpha \in v_jQ_1v_k$, $[(v_k - (-1)^{k-j-1}v_j)\alpha](A) = 0$. But this is clear, since $v_j(A) = (-1)^{k-j-1}v_k(A)$. Finally, suppose that $\alpha \in Q_1$ is a loop at $v$. If $(v_j,v_k) \neq (v,v)$ then of course $[v\alpha](A) = 0$. Otherwise $(v_j,v_k) = (v,v)$. In this case, $v \in L(A)$, as noted above, so that $v(A) = 0$ and hence $[v\alpha](A) = 0$. We conclude that every element of $X_{1/2}$ vanishes on $\msa (Q)$. A similar argument shows that every element of $X_1$ also vanishes on $\msa (Q)$.

All that remains to be shown is that any homogeneous polynomial vanishing on $\msa (Q)$ belongs to $X$. Let $f \in R$ be a homogeneous polynomial of degree $d$ vanishing on $\msa (Q)$. Then we can write $f = \sum_{\sigma}{\phi_{\sigma}\mu_{\sigma}}$, where: 

\begin{enumerate}
\item the sum runs through all functions $\sigma : Q_0\setminus\{v_0\}\cup Q_1 \rightarrow \ZZ_{\geq 0}$ satisfying $|\sigma| = d$, 
\item each $\phi_{\sigma}$ is an element of $k$, and 
\item \begin{equation*}
\mu_{\sigma} := \prod_{x \in Q_0\setminus\{v_0\}\cup Q_1}{x^{\sigma(x)}}. 
\end{equation*} 
\end{enumerate}

\noindent Then for all $1 \le i \le n_0$, $0 = f(A(v_0+v_i)) = \sum_{\sigma}{\phi_{\sigma}\mu_{\sigma}(A(v_0+v_i))} = \phi_{d\cdot v_i^*}v_i(A(v_0+v_i))^d$. There is a frame for $L(A(v_0+v_i))$ in which $v_i(A(v_0+v_i)) = 1$, and so this implies $\phi_{d\cdot v_i^*} = 0$ for all $1 \le i \le n_0$.   

For all $1 \le j < k \le n_0$,  
\begin{equation*}
0 = f(A(v_j+v_k)) = \sum_{i=1}^{d-1}{\phi_{i\cdot v_j^* + (d-i)\cdot v_k^*}v_j(A(v_j+v_k))^iv_k(A(v_j+v_k))^{d-i}}.
\end{equation*}
\noindent Note that the previous argument allowed us to exclude $i = 0$ and $i = d$. Since we can choose a frame for $L(A(v_j+v_k))$ in which $v_j(A(v_j+v_k)) = (-1)^{k-j-1}$ and $v_k(A(v_j+v_k)) = 1$, we see that this relation reduces to 
\begin{equation*} 
0 = \sum_{i=1}^{d-1}{\phi_{i\cdot v_j^* + (d-i)\cdot v_k^*}(-1)^{(k-j-1)i}},
\end{equation*} 
\noindent which in turn holds if and only if  

\begin{equation*}
\sum_{i=1}^{d-1}{\phi_{i\cdot v_j^*+(d-i)\cdot v_k^*}\mu_{i\cdot v_j^* + (d-i)\cdot v_k^*}} \in X_0. 
\end{equation*}
 \noindent Every other term $\phi_{\sigma}\mu_{\sigma}$ in $f$ satisfying $\sigma (Q_1) = 0$ must involve three distinct vertices, and all such terms already belong to $X_0$. So without loss of generality, we may assume that $f$ is contained in the ideal of $R$ generated by $Q_1$. Of course, we may also subtract off any terms in $f$ that are already in $X_{1/2}$ or $X_1$. So in fact, we may assume that if $\phi_{\sigma} \neq 0$, then there exist $v_j,v_k \in Q_0\setminus\{v_0\}$ such that $\sigma(v_jQ_1v_k) \neq 0$ and $\supp (\sigma) \subset \{ v_j, v_k\} \cup v_jQ_1v_k$. Suppose that $v_jQ_1v_k = \{ \alpha_i, \ldots , \alpha_q\}$. Let $\lambda$ be a function $\{j\} \cup \{ i,\ldots , q -1\} \rightarrow \KK^{\times}$, and let $A(\lambda)$ be the maximal subalgebra with basis $Q_0\setminus\{ v_j,v_k\}\cup \{ v_j + \lambda(j) \alpha_q, v_k-\lambda(j) \alpha_q\} \cup Q_1\setminus v_jQ_1v_k \cup \{ \alpha_i+\lambda(i)\alpha_q, \ldots , \alpha_{q-1} + \lambda(q-1)\alpha_q\}$ (again, if $v_j = v_k$ we assume $\lambda(j) = 0$). Then evaluating at the corresponding frame yields  

\begin{equation*} 
0 = f(A(\lambda)) = \sum{\phi_{q\cdot v_j^* + r\cdot v_k^* + \tau}(\lambda(j)(-1)^{n_0+q-j})^q(\lambda(j)(-1)^{n_0+q-k+1})^r\lambda^{\tau}},
\end{equation*}

\noindent where: 

\begin{enumerate}
\item the sum ranges over all triples $(q,r,\tau )$ such that $\supp (\tau) \subset v_jQ_1v_k$ and $|\tau| = d-q-r$, and 
\item 
\begin{equation*}
\lambda^{\tau}:= \prod_{i \le \ell \le q-1}{\alpha_{\ell}(A(\lambda))^{\tau(\alpha_{\ell})}} = \pm \prod_{i \le \ell \le q-1}{\lambda(\ell)^{\tau (\alpha_{\ell})}}.  
\end{equation*} 
\end{enumerate}

\noindent Since this must hold for all $\lambda$, it follows that for all fixed $\tau$, 

\begin{equation*} 
0 = \sum_{i=0}^{d - |\tau |}{\phi_{(d-|\tau | - i)\cdot v_j^* + i\cdot v_k^* + \tau}(-1)^{(1-j-k)i}},
\end{equation*} 
\noindent where we have gathered terms from the powers of $v_j(A(\lambda))$ and $v_k(A(\lambda) )$, and have cancelled constant powers of $-1$ from both sides of the equation. When $|\tau | < d$, this implies that the sum of all monomials in $f$ containing a vertex and and arrow lie in $X_0 + X_{1/2}$. Hence, we have reduced to the following case: $f = \sum{\phi_{\sigma}\mu_{\sigma}}$, where $\phi_{\sigma} \neq 0$ implies that for some $v_j, v_k \in Q_0$, $\supp (\sigma ) \subset v_jQ_1v_k$. But applying $A(\lambda )$ to such a polynomial yields 

\begin{equation*} 
0 = \sum_{\supp(\sigma )\subset v_jQ_1v_k}{\phi_{\sigma}\mu_{\sigma}(A(\lambda))}.
\end{equation*} 

\noindent Notice that  

\begin{equation*}
\mu_{\sigma}(A(\lambda)) = \pm\prod_{\ell \in \{j\} \cup \{i,\ldots , q\}}{\lambda(\ell)^{\sigma(\alpha_{\ell})}}.  
\end{equation*}

\noindent Since this holds for all such $\lambda$, it follows that for all $\sigma$ with $\supp (\sigma) \subset v_jQ_1v_k$, $\phi_{\sigma} = 0$. Applying this to all pairs of vertices $v_j,v_k \in Q_0$, we conclude that $f \equiv 0$. In other words, any polynomial vanishing on $\msa (Q)$ must lie in $X_0 + X_{1/2} + X_1 = X$, as we wished to show.
\end{proof}     

\noindent For the following Corollary, we recall the definitions of some standard quivers:  

\begin{definition}\label{d.irredQ}
For any $m\in \NN$, the \emph{$m$-loop quiver} is the quiver with vertex set $\{ 1\}$ and arrow set $\{\alpha_1,\ldots , \alpha_m\}$: 
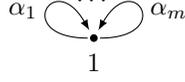
\[ 
\vcenter{\hbox{ 
 \begin{tikzpicture}[point/.style={shape=circle,fill=black,scale=.3pt,outer sep=3pt},>=latex, baseline=-3,scale=2]  
\node[point,label={below:$1$}] (1) at (0,0) {} ;   
\node (2) at (0,.25) {$\cdots$} ;
\path[->] (1) edge[in = 120, out= 180, loop] node[left] {$\alpha_1$} (1) edge[in = 60, out = 0, loop] node[right] {$\alpha_m$}  (1);
\end{tikzpicture}
}}
\]

\noindent We denote this quiver by $\mathbb{L}_m$. The \emph{$m$-Kronecker quiver} is the quiver with vertex set $\{s,t\}$ and arrow set $Q_1 = sQ_1t = \{ \alpha_1,\ldots , \alpha_m\}$: 

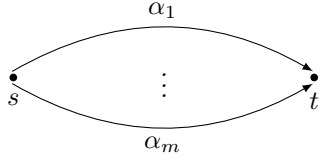
\[ 
\vcenter{\hbox{ 
 \begin{tikzpicture}[point/.style={shape=circle,fill=black,scale=.3pt,outer sep=3pt},>=latex, baseline=-3,scale=2]   
 \node[point,label={below:$s$}] (1) at (0,0) {} ; 
 \node[point,label={below:$t$}] (2) at (2,0) {} ;  
 \node (3) at (1,0) {$\vdots$} ;
 \path[->] (1.100) edge[bend left] node[above] {$\alpha_1$} (2.100)  
 (1.-100) edge[bend right] node[below] {$\alpha_m$} (2.-100); 
 
\end{tikzpicture} 
}} 
\] 

\noindent We denote this quiver by $\mathbb{K}_m$. Finally, the \emph{$m$ isolated vertices} quiver is just the quiver with vertex set $\{ 1,\ldots , m\}$ and arrow set $Q_1 = \emptyset$.
\end{definition}

\begin{example}\label{e.irred}
It is easy to check that any subspace of $k\mathbb{K}_m$ is a subalgebra if and only if it contains $1$. Hence, $\msa (\mathbb{K}_m) = \Gr_{m}(k\mathbb{K}_m/k\cdot 1) \cong \mathbb{P}^{m}$. Similarly, any subspace of $J(T_2\mathbb{L}_m)$ is multiplicatively closed, and so $\msa (\mathbb{L}_m) = \Gr_{m-1}(J(T_2\mathbb{L}_m)) \cong \mathbb{P}^{m-1}$. If $Q$ is two isolated vertices, then $k\cdot 1$ is the only proper subalgebra and $\msa (Q)$ is a point. It turns out that for all other $Q$, $\msa (Q)$ is reducible.
\end{example}

\begin{corollary}\label{c.irred}
Let $Q$ be a finite quiver. Identify $\msa (Q)$ with the variety of maximal subalgebras of $T_2Q$. Then the irreducible components of $\msa (Q)$ may be described explicitly as follows: 
\begin{enumerate} 
\item For each two-element subset $\{ s,t\} \subset Q_0$ such that $sQ_1t \cup tQ_1s = \emptyset$, the singleton $\{ A(s+t)\}$ is an irreducible component of dimension $0$;  
\item For each $(s,t) \in V^2(Q)$ with $s \neq t$, the collection  
\begin{equation*}
\left(\bigcup_{U \in \PP ((skQ_1t)^*)}{[\Inn^*(B)\cdot A(s,t,U)]} \right) \cup \{ A(s+t)\} 
\end{equation*}
\noindent is an irreducible component of dimension $|sQ_1t|$; 
\item For each $(s,s) \in V^2(Q)$, the collection  
\begin{equation*}
\bigcup_{U \in \PP ((skQ_1s)^*)}{[\Inn^*(B)\cdot A(s,s,U)]} 
\end{equation*}
\noindent is an irreducible component of dimension $|sQ_1s| -1$.
\end{enumerate} 

\noindent These irreducible components are all projective spaces, and the dimension of $\msa (Q)$ is the maximum of the dimensions described above. Furthermore, $\msa (Q)$ is an irreducible projective variety if and only if $Q$ is an $m$-Kronecker quiver, an $m$-loop quiver, or the two isolated vertices quiver. 
\end{corollary} 

\begin{proof} 
It is clear that the singletons in (1) are closed irreducible subsets of $\msa (Q)$. So suppose $(v_j, v_{\ell}) \in V^2(Q)$ is given with $j \neq \ell$. Assume $j<\ell$. We claim that  

\begin{equation*}
\bigcup_{U \in \PP ((v_jkQ_1v_{\ell})^*)}{\Inn^*(B)A(v_j,v_{\ell},U)} \cup \{ A(v_j+v_{\ell})\} 
\end{equation*}

\noindent is the vanishing set of a homogeneous prime ideal $P$. Indeed, it is easy to check using the proof of theorem \ref{t.msa(Q)} that $P = X+(Q_0\setminus\{v_0, v_{\ell}\}, Q_1\setminus v_jQ_1v_{\ell})$ if $j= 0$, and $P = X+(v_j - (-1)^{\ell-j-1}v_{\ell}, Q_0\setminus\{ v_0,v_j,v_{\ell}\}, Q_1\setminus v_jQ_1v_{\ell})$ otherwise. In either case, we find that $k[\msa (Q)]/P$ is isomorphic to the polynomial ring $k[v_{\ell}, v_jQ_1v_{\ell}]$. The case $\ell < j$ is similar. Finally, suppose that $v_j \in Q_0$ satisfies $ v_jQ_1v_j \neq \emptyset$. Then $\bigcup_{U \in \PP ((v_jkQ_1v_j)^*)}{\Inn^*(B)A(v_j,v_j,U)}$ is the vanishing set of $P = X+(Q_0\setminus \{ v_0\} , Q_1\setminus v_jQ_1v_j)$ and $k[\msa (Q)]/P \cong k[v_jQ_1v_j]$. 
 
 Every element of $\msa (Q)$ is contained in the union of the closed irreducible subsets described above. Since there are no containment relations between these sets, it follows that these are the irreducible components of $\msa (Q)$. From our description, everything is clear but the criterion for when $\msa (Q)$ is itself irreducible. In example \ref{e.irred}, we verified that $\msa (Q)$ is irreducible when $Q$ is an $m$-Kronecker quiver, an $m$-loop quiver, or the union of two isolated vertices. So we only need to show that for every other $Q$, $\msa (Q)$ is reducible. If $Q$ is not connected, and is not the union of two isolated vertices, then (1)-(3) yield more than one irreducible component, so we may assume that $Q$ is connected. If there exist non-parallel arrows in $Q$, then $X_1$ is not the zero ideal, and $k[\msa (Q)]$ is not a domain. So all arrows in $Q$ must be parallel, which implies that either $Q$ is a Kronecker quiver or an $n$-loop quiver.
\end{proof}    

\begin{example} 
Let  

\[  
Q=
\vcenter{\hbox{ 
 \begin{tikzpicture}[point/.style={shape=circle,fill=black,scale=.3pt,outer sep=3pt},>=latex, baseline=-3,scale=2]   
 \node[point,label={below:$s$}] (1) at (0,0) {} ; 
 \node[point,label={below:$t$}] (2) at (2,0) {} ;  
 \path[->] (1.100) edge node[above] {$\alpha_1$} (2.100)  
 (1.-50) edge node[below] {$\alpha_2$} (2.-125)
 (2.-100) edge[bend left] node[below] {$\beta$} (1.-100) 
 (2) edge[in=45, out=-45,loop] node[right] {$\gamma$} () ;
\end{tikzpicture} 
}} 
\]  
\noindent and take $R = k[t,\alpha_1,\alpha_2,\beta,\gamma]$. Then $\msa (Q)$ is the vanishing set of the ideal $X = (\gamma t, \alpha_1\beta, \alpha_2\beta, \alpha_1\gamma, \alpha_2\gamma, \gamma \beta)$. $V^2(Q)$ consists of the three points $ (s,t)$, $(t,s)$, and $(t,t)$. Each of these corresponds to an irreducible component of dimension $2$, $1$, and $0$, respectively.
\end{example}

\section{Orbits and Isoclasses in $\msa (Q)$}\label{s.iso}     

Let $B = kQ/I$ for some admissible ideal $I$. Taking the Wedderburn-Malcev decomposition $B = kQ_0 \oplus J(Q)/I$ for $B$ discussed in section \ref{s.bg}, we define the following subgroups of $\Aut_k(B)$: 

\begin{equation*} 
\hat{H}_B = \{ \phi \in \Aut_k(B)\mid \phi (Q_0) = Q_0\}, 
\end{equation*} 

\begin{equation*} 
H_B = \{ \phi \in \Aut_k(B) \mid \phi\mid_{Q_0} = \id_{Q_0} \}, 
\end{equation*} 

\begin{equation*} 
\operatorname{Vl}_{(Q,I)} = \{ \phi \in H_B \mid \phi (ukQ_1v) = ukQ_1v \text{ for all $(u,v) \in V^2(Q)$} \}.
\end{equation*}

\noindent These definitions are chosen to be consistent with the terminology of \cite{GS1}. The group $\operatorname{Vl}_{(Q,I)}$ is called the group of \emph{linear changes of variables}. In definition \ref{msaDef}, we note that the action of $\Aut_k(B)$ on $\msa (Q)$ factors through the map $\Aut_k(B) \rightarrow \Aut_k(B/J(B)^2) = \Aut_k(T_2Q)$. Hence, $\Aut_k(B)$ acts on $\msa (Q)$ through a subgroup of $\Aut_k(T_2Q)$. This allows us to think of $\Aut_k(T_2Q)$ as a collection of ``extra symmetries'' of $\msa (Q)$, beyond those afforded by $\Aut_k(B)$. Thankfully, the structure theory of $\Aut_k(T_2Q)$ is fairly straight-foward, so we recall it below.

 Write $G = H_{T_2Q}$, and consider $S_Q$ as a subgroup of $\Aut_k(T_2Q)$. Then there exist short exact sequences  

\begin{equation}\label{e.ses1}
1 \rightarrow \Inn^*(T_2Q) \rightarrow \Aut_k(T_2Q) \rightarrow \Aut_k(T_2Q)/\Inn^*(T_2Q) \rightarrow 1,
\end{equation}  

\begin{equation}\label{e.ses2}
1 \rightarrow G \rightarrow \Aut_k(T_2Q)/\Inn^*(T_2Q) \rightarrow S_Q \rightarrow 1,
\end{equation}  

\noindent both of which split on the right. First we show that, as an abstract group, $\Aut_k(T_2Q)$ is generated by $\Inn^*(T_2Q)$, $G$, and $S_Q$. Let $\phi : T_2Q \rightarrow T_2Q$ be an automorphism. Then there exists an inner automorphism $ \iota_u \in \Inn(T_2Q)$ such that $\psi:= \iota_u\phi$ maps $Q_0$ to itself (see, for instance \cite{IS}). Furthermore, if we write $u = \sum_{v \in Q_0}{\lambda_vv} + \sum_{\alpha \in Q_1}{\lambda_{\alpha}\alpha}$, then $u = (\sum_{v\in Q_0}{\lambda_vv})(1 + \sum_{\alpha \in Q_1}{\frac{\lambda_{\alpha}}{\lambda_{s(\alpha)}}\alpha})$, and $\iota_{\sum_{v \in Q_0}{\lambda_vv}}$ fixes each vertex. Hence, we may assume that $u= 1+x$, where $x \in J(T_2Q) = kQ_1$. Then for all $s, t \in Q_0$, $\psi$ maps the standard basis of $skQ_1t$ to a basis of $\psi(s)kQ_1\psi(t)$. Therefore, there exists a $g \in G$ such that $\tau:= g\psi$ maps the arrows from $s\rightarrow t$ to the arrows $\psi (s) \rightarrow \psi (t)$. In other words, $\tau \in \Aut(Q)$. However, $\tau$ can be written as $\tau = \tau_o\circ \tau_g$, where $\tau_o \in S_Q$ and $\tau_g \in G$. This shows that $\Aut_k(T_2Q)$ is generated by $\Inn^*(T_2Q)$, $G$, and $S_Q$. 

To verify sequence (1), suppose that $\iota_{1+x} \in \Inn^*(T_2Q) \cap G\cdot S_Q$. Since $\iota_{1+x}(v) \equiv v$ $(\operatorname{mod}$ $kQ_1)$ for all $v \in Q_0$, we must have $\iota_{1+x} \in G$. But then $\iota_{1+x}(v) = v$ for all $v \in Q_0$, and $\iota_{1+x}(\alpha) = \alpha$ for all $\alpha \in Q_1$. In other words, $\iota_{1+x} = \id$, so $\Aut_k(T_2Q)/\Inn^*(T_2Q) \cong G\cdot S_Q$. This implies that sequence (1) splits on the right. For sequence (2), simply note that $G$ is normal in $G\cdot S_Q$, and that $G\cap S_Q = \{ \id \}$.  

The subgroups $G$ and $\Inn^*(T_2Q)$ are easy to explicitly describe. We start with $G$: since an element of $G$ fixes each vertex, it sends $ukQ_1v$ to itself for all $u,v \in Q_0$. The restriction to $ukQ_1v$ is just an invertible linear map. Hence, there is an inclusion 

\begin{equation*}
 G \hookrightarrow \prod_{u,v \in Q_0}{\GL(ukQ_1v)}, 
 \end{equation*}
 \noindent which is surjective since $T_2Q$ has radical square zero. For $\Inn^*(T_2Q)$, it is straightforward to check that the map $x \mapsto \iota_{1+x}$ is an isomorphism of groups $(kQ_1)_a \cong \Inn^*(T_2Q)$, where $(kQ_1)_a$ denotes the additive group of $kQ_1$. Now we can reduce the problem of describing $\Aut_k(T_2Q)$-orbits on $\msa (Q)$ to the action of the finite group $S_Q$ on certain finite sets: 
 
 \begin{proposition}\label{t.J2struc}   
  $\Aut_k(T_2Q)$-orbits on $\msa (Q)$ may be described as follows: 
\begin{enumerate} 
\item Orbits containing a maximal subalgebra of split type are in bijection with $S_Q$-orbits on $V^2(Q)$. 
\item Orbits containing a maximal subalgebra of separable type are in bijection with $S_Q$-orbits on $\{ \{s,t\} \subset Q_0 \mid s \neq t \}$. 
\end{enumerate} 
 
 \end{proposition} 
 
 \begin{proof} 
 $\Inn^*(T_2Q)$-orbits on $\msa (Q)$ are already described by theorem \ref{t.basic}, so it remains to see which $\Inn^*(T_2Q)$-orbits lie in the same $G$-orbit, and which of those lie in the same $S_Q$-orbit. Any element $g \in G$ fixes the maximal subalgebras of separable type, and sends $A(s,t,V)$ to $A(s,t, gV)$. Since the action of $\GL (skQ_1t)$ on $\PP ((skQ_1t)^*)$ is transitive, it follows that $A(s,t,U)$ and $A(s,t,W)$ lie in the same $\Aut_k(T_2Q)$-orbit, for any two $U,W \in \PP((skQ_1t)^*).$ So the $\Aut_k(T_2Q)$ orbits are determined by the action of $S_Q$ on vertices. For $\sigma \in S_Q$, $\sigma(A(s,t,U)) = A(\sigma (s), \sigma (t), \sigma (U))$ and $\sigma (A(u+v)) = A(\sigma (u) + \sigma (v))$. Claims (1) and (2) are now clear.

 \end{proof}

\noindent Let $B = kQ/I$. Then the map $H_B \rightarrow H_{T_2Q}  = \prod_{u,v \in Q_0}{\GL(ukQ_1v)}$ has closed image. Let $H_B(u,v)$ denote the image of the corresponding map $\pi_{uv}: H_B \rightarrow H_{T_2Q} \rightarrow \GL(ukQ_1v)$.

\begin{theorem}\label{t.orb}
Let $B = kQ/I$. Then there are finitely many $\Aut_k(B)$-orbits on $\msa (Q)$ if and only if for each $(u,v) \in V^2(Q)$, there are finitely-many $H_B(u,v)$-orbits on $\Gr_{d(u,v)-1}(ukQ_1v)$.
\end{theorem} 

\begin{proof} 
There are ${|Q_0|}\choose{2}$ orbits corresponding to maximal subalgebras of separable type, so it suffices to prove the equivalence on maximal subalgebras of split type. To prove sufficiency, suppose that $A(u,v,W)$ is a maximal subalgebra of split type, and that for all $(u,v) \in V^2(Q)$, $\Gr_{d(u,v)-1}(ukQ_1v)$ is a finite union of $H_B(u,v)$-orbits. Consider $A(u,v,\bar{h}W)$, where $\bar{h} $ is an element of $H_B(u,v)$. Picking $h \in H_B$ with $\bar{h} = \pi_{uv}(h)$, we see that $A(u,v,\bar{h}W) = A(u,v, \pi_{uv}(h)W) = hA(u,v,W)$. In other words, $A(u,v,W)$ is conjugate to $A(u,v,\bar{h}W)$. This implies that there are finitely-many orbits in $\msa (Q)$ under the action of the group generated by $\Inn^*(B)$ and $H_B$.  

It remains to prove necessity. First, note that  
\begin{equation*}
\Aut_k(B) = \Inn^*(B)\hat{H}_B = \hat{H}_B\Inn^*(B). 
\end{equation*}
 Suppose that $A_1$ and $A_2$ are maximal subalgebras of split type lying in the same $\Aut_k(B)$ orbit. We know that there exist triples $(u_1,v_1,W_1)$ and $(u_2,v_2,W_2)$ such that $A_i$ is $\Inn^*(B)$-conjugate to $A(u_i,v_i,W_i)$, for $i = 1,2$. So we may assume $A(u_2,v_2,W_2) = \iota \hat{h} A(u_1,v_1,W_1)$, for some $\iota \in \Inn^*(B)$ and $\hat{h} \in \hat{H}_B$. Then $\hat{h}A(u_1,v_1,W_1)$ is $\Inn^*(B)$-conjugate to $A(u_2,v_2,W_2)$, and so their Jacobson radicals coincide. If we write $J(u,v,W)$ for the Jacobson radical of $A(u,v,W)$, then this means $\hat{h}J(u_1,v_1,W_1) = J(u_2,v_2,W_2)$. Conversely, $\hat{h}J(u_1,v_1,W_1) = J(u_2,v_2,W_2)$ implies $\hat{h}A(u_1,v_1,W_1) = A(u_2,v_2,W_2)$, since $\hat{h}$ permutes $Q_0 \subset A(u_1,v_1,W_1)\cap A(u_2,v_2,W_2)$. It follows that $A(u_1,v_1,W_1)$ lies in the same $\Aut_k(B)$-orbit of $A(u_2,v_2,W_2)$, if and only if $J(u_1,v_1,W_1)$ and $J(u_2,v_2,W_2)$ lie in the same $\hat{H}_B$-orbit of codimension-$1$ $B$-subbimodules of $J(B)$. Hence, there are finitely-many $\hat{H}_B$-orbits of codimension-$1$ $B$-subbimodules of $J(B)$. But $H_B$ is a finite-index normal subgroup of $\hat{H}_B$ by proposition 13 of \cite{GS1}, and so each $\hat{H}_B$ orbit is a finite union of $H_B$-orbits. Furthermore, the $H_B$-orbits of codimension-$1$ subbimodules can be expressed as the union of the $H_B(u,v)$-orbits of $\Gr_{d(u,v)-1}(ukQ_1v)$ over the pairs $(u,v) \in V^2(Q)$. So for each fixed $(u,v) \in V^2(Q)$, $\Gr_{d(u,v)-1}(ukQ_1v)$ is a finite union of $H_B(u,v)$-orbits.
\end{proof} 

\begin{corollary} 
Let $B = kQ/I$. Then $\msa (Q)$ is a finite union of $\Aut_k(B)$-orbits if one of the following conditions holds: 
\begin{enumerate} 
\item $Q$ is Schur. 
\item $I$ is monomial. 
\end{enumerate}
\end{corollary}  

\begin{proof} 
If $Q$ is Schur, then $d(u,v) = 1$ for all $(u,v) \in V^2(Q)$ and each Grassmannian $\Gr_{d(u,v)-1}(ukQ_1v)$ is a point. If $I$ is monomial, then for all $(u,v) \in V^2(Q)$, $H_B(u,v)$ contains the invertible diagonal matrices, which decompose $ukQ_1v$ into finitely-many orbits.
\end{proof}   

\noindent Of course, these conditions are not necessary for $\msa (Q)$ to be a finite union of $\Aut_k(B)$-orbits. We now construct a class of (generally) non-Schur, non-monomial $B$ which decompose $\msa (Q)$ as a finite union of $\Aut_k(B)$-orbits. 

\begin{definition} 
Let $Q$ be a quiver. Then we define $r_2(Q)$ to be the quiver with vertex set $(r_2Q)_0 = V^2(Q)$, for which there is an arrow $(u,v) \xrightarrow[]{\alpha} (u',v')$ if and only if $v=v'$. 
\end{definition} 

\begin{definition} 
Let $\lambda \in kQ$ be any element. We call $\lambda$ an \emph{$r_2$-element} if there exist $(u,v) , (v,w) \in V^2(Q)$ such that $\lambda = \sum_{\alpha, \beta}{\lambda_{\alpha \beta}\alpha\beta}$, where the sum ranges over all $\alpha \in uQ_1v$ and $\beta \in vQ_1w$, and $\lambda_{\alpha\beta} \in k$. If $(u,v) \xrightarrow[]{\alpha} (v,w)$ is the corresponding arrow in $r_2Q$, then we say that $\lambda$ has \emph{type $\alpha$}.
\end{definition}  

\noindent {\bf{Note:}} For any quiver $Q$, $r_2(Q)$ is always a Schur quiver. For any arrow $\alpha \in r_2Q$, $0$ is an $r_2$-element of type $\alpha$.

\begin{definition} 
An admissible ideal $I \subset kQ$ is called an \emph{$r_2$-ideal} if it is generated by a set $\{ \lambda_{\alpha} \mid \alpha \in (r_2Q)_1 \}$ such that for all arrows $\alpha \in r_2Q$, $\lambda_{\alpha}$ is an $r_2$-element of type $\alpha$.
\end{definition}  

\noindent Let $I$ be an $r_2$-ideal of $kQ$. Fix a generating set $\lambda = \{ \lambda_{\alpha} \mid \alpha \in (r_2Q)_1 \}$ of $I$, such that each $\lambda_{\alpha}$ is an $r_2$-element of type $\alpha$. If $\alpha$ is an arrow from $(u,v)$ to $(v,w)$, write $\lambda_{\alpha} = \sum_{\gamma, \delta}{\lambda_{\gamma \delta}\gamma \delta}$, for $\gamma \in uQ_1v$ and $\delta \in vQ_1w$. Then $\lambda$ gives rise to a representation of $r_2Q$, which we term $V_{\lambda}$. To each $(u,v) \in V^2(Q)$, we attach the vector space $V_{\lambda}(u,v):= ukQ_1v$. The map $V_{\lambda}(u,v) \xrightarrow[]{V_{\lambda}(\alpha)} V_{\lambda}(u,v)$ is then defined by 

\begin{equation*} 
\gamma \mapsto \sum_{\delta \in vQ_1w}{\lambda_{\gamma \delta}{\delta}}.
\end{equation*}    

\noindent We let $\underline{d}$ denote the dimension vector of this representation, i.e. $\underline{d} : V^2(Q) \rightarrow \NN$ is the function $\underline{d}(u,v) = d(u,v) = \dim_kukQ_1v$.  \newline

\noindent {\bf{Note:}} To be consistent with the multiplication in the path algebra, these representations should be considered \emph{right} $r_2(Q)$-modules. In the example below, therefore, all vector spaces should be understood as row vectors, with matrices acting on the right.

\begin{example}\label{r2Ex}
Let $Q$ be the quiver 

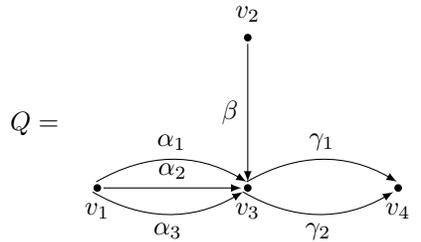
\[  
Q=
\vcenter{\hbox{ 
 \begin{tikzpicture}[point/.style={shape=circle,fill=black,scale=.3pt,outer sep=3pt},>=latex, baseline=-3,scale=2]   
 \node[point,label={below:$v_1$}] (1) at (0,0) {} ; 
 \node[point,label={above:$v_2$}] (2) at (1,1) {} ;   
 \node[point,label={below:$v_3$}] (3) at (1,0) {} ; 
 \node[point,label ={below:$v_4$}] (4) at (2,0) {} ;
 \path[->] (1.100) edge[bend left] node[above] {$\alpha_1$} (3.100)   
 (1) edge node[above] {$\alpha_2$} (3) 
 (1.-500) edge[bend right] node[below] {$\alpha_3$} (3.-500)
 (2) edge node[left] {$\beta$} (3)
 (3.100) edge[bend left] node[above] {$\gamma_1$} (4.100) 
 (3.-500) edge[bend right] node[below] {$\gamma_2$} (4.-500) ;
\end{tikzpicture}.
}} 
\]   

\noindent Then $r_2(Q)$ is the quiver 

\[  
r_2(Q)=
\vcenter{\hbox{ 
 \begin{tikzpicture}[point/.style={shape=circle,fill=black,scale=.3pt,outer sep=3pt},>=latex, baseline=-3,scale=2]   
 \node[point,label={below:$(v_2,v_3)$}] (1) at (0,0) {} ; 
 \node[point,label={below:$(v_3,v_4)$}] (2) at (1,0) {} ;   
 \node[point,label={below:$(v_1,v_3)$}] (3) at (2,0) {} ; 

 \path[->] (1) edge (2) 
 (3) edge (2) ;
\end{tikzpicture}.
}} 
\]    
Consider the ideals $I_1 = (\alpha_1\gamma_1-\alpha_1\gamma_2-\alpha_2\gamma_1+\alpha_2\gamma_2 + \pi\alpha_3\gamma_2)$, $I_2 = (\alpha_1\gamma_2 + \alpha_2\gamma_1, \beta\gamma_1)$, $I_3 = (\alpha_1\gamma_2+\alpha_2\gamma_1, \alpha_1\gamma_1+\alpha_2\gamma_2)$. Then $I_1$ and $I_2$ are $r_2$-ideals of $kQ$, but $I_3$ is not. The representation corresponding to the $r_2$-generating set $\lambda = \{ \alpha_1\gamma_1-\alpha_1\gamma_2-\alpha_2\gamma_1+\alpha_2\gamma_2 + \pi\alpha_3\gamma_2 \}$ is 

\[  
V_{\lambda}=
\vcenter{\hbox{ 
 \begin{tikzpicture}[point/.style={shape=circle,fill=black,scale=.3pt,outer sep=3pt},>=latex, baseline=-3,scale=2]   
 \node[point,label={below:$k$}] (1) at (-1,0) {} ; 
 \node[point,label={below:$k^2$}] (2) at (1,0) {} ;   
 \node[point,label={below:$k^3$}] (3) at (3,0) {} ; 

 \path[->] (1) edge node[above] {$\left(\begin{array}{cc} 0 & 0 \end{array} \right)$
} (2) 
 (3) edge node[above] {$\left( \begin{array}{cc} 1 & -1 \\ -1 & 1 \\ 0 & \pi \end{array} \right)$} (2) ;
\end{tikzpicture},
}} 
\]    

\noindent and the representation corresponding to the $r_2$-generating set $\mu = \{\alpha_1\gamma_2 + \alpha_2\gamma_1, \beta\gamma_1\}$ is 

\[  
V_{\mu}=
\vcenter{\hbox{ 
 \begin{tikzpicture}[point/.style={shape=circle,fill=black,scale=.3pt,outer sep=3pt},>=latex, baseline=-3,scale=2]   
 \node[point,label={below:$k$}] (1) at (-1,0) {} ; 
 \node[point,label={below:$k^2$}] (2) at (1,0) {} ;   
 \node[point,label={below:$k^3$}] (3) at (3,0) {} ; 

 \path[->] (1) edge node[above] {$\left(\begin{array}{cc} 1 & 0 \end{array} \right)$
} (2) 
 (3) edge node[above] {$\left( \begin{array}{cc} 0 & 1 \\ 1 & 0 \\ 0 & 0 \end{array} \right)$} (2) ;
\end{tikzpicture}.
}} 
\]  
\end{example}

\noindent Let $B = kQ/I$. Then a linear change of variables $g \in \operatorname{Vl}_{(Q,I)} \le H_B$ can be identified with a an element $g = (g_{uv}) \in \prod_{u,v \in Q_0}{\GL(ukQ_1v)}$ satisfying $g(I) = I$. Note that $g\lambda_{\alpha}$ is an $r_2$-element of type $\alpha$, and that since $I$ is an $r_2$-ideal, the condition $g(I) = I$ means $g\lambda_{\alpha} \in k\lambda_{\alpha} \setminus \{ 0 \}$. From this, it is easy to see that this is equivalent to $g_{uv}V_{\lambda}(\alpha)g_{vw}^T = \chi(\alpha) V_{\lambda}(\alpha)$, for some $\chi(\alpha) \in k^{\times}$ (where $T$ denotes the usual transpose of a matrix). We are now ready to build our examples. For the next proposition, we introduce the following notation: if $W$ is a vector space, then we let $\operatorname{O}(W)$ denote the orthogonal group of $W$, i.e. the group of all matrices $g$ satisfying $g^Tg = I$. 

\begin{proposition}\label{r2FOP}
Let $Q$, $I$, $B$, and $\lambda$ be as above. Suppose that the no-double-edge graph of $Q$ is a disjoint union of oriented trees. Suppose further that there exists a $\phi = (\phi_{uv}) \in \prod_{u,v \in Q_0}{\operatorname{O}(ukQ_1v)} \le \prod_{u,v \in Q_0}{\GL (ukQ_1v)}$, such that the (isomorphic) representation $W:= \phi V_\lambda$ satisfies the following: for each $\alpha \in (r_2Q)_1$, there is at most one non-zero entry in each row and column of $W(\alpha )$. Then $\msa (Q)$ is a finite union of $\Aut_k(B)$-orbits.
\end{proposition} 

\begin{proof} 
By theorem \ref{t.orb}, it suffices to show that for all $(u,v) \in V^2(Q)$, $ukQ_1v$ is a finite union of $H_B(u,v)$-orbits. We show that in fact, $ukQ_1v$ is a finite union of $\pi_{uv}(\operatorname{Vl}_{(Q,I)})$-orbits. Since the no-double-edge graph of $Q$ is a disjoint union of oriented trees, so is $r_2Q$. As before, $\operatorname{Vl}_{(Q,I)}$ can be identified with the elements $g = (g_{uv}) \in \prod_{u,v \in Q_0}{\GL(ukQ_1v)}$ satisfying $g_{uv}V_{\lambda}(\alpha)g_{vw}^T = \chi(\alpha) V_{\lambda}(\alpha)$ for all arrows $\alpha \in r_2Q$. Now, $W(\alpha) = \phi_{uv}V_{\lambda}(\alpha)\phi_{vw}^{-1}$ for all $\alpha \in (r_2Q)_1$. Since $\phi_{uv}^{-1} = \phi_{uv}^T$ for all $(u,v) \in V^2(Q)$, $\operatorname{Vl}_{(Q,I)}$ is conjugate in $\prod_{u,v\in Q_0}{\GL(ukQ_1v)}$ to the collection of all $g$ satisfying $g_{uv}W(\alpha)g_{vw}^T = \chi(\alpha) W(\alpha)$ for all $\alpha$, and some $\chi(\alpha) \in k^{\times}$. Call this group $G$. Pick a function $\mu : uQ_1v \rightarrow k^{\times}$, and let $g_{uv}^{\mu} \in \GL (ukQ_1v)$ be the map $g_{uv}^{\mu}(\gamma ) = \mu(\gamma )\gamma$. We claim that there exists a $g \in G$ such that $g_{uv} = g_{uv}^{\mu}$. To construct this map, pick $\alpha \in (r_2Q)_1$ with $s(\alpha) = (u,v)$. Let $t(\alpha) = (v,w)$. Then $g_{uv}^{\mu}W(\alpha)$ rescales the rows of $W(\alpha)$, and so there exists an appropriate diagonal matrix $g_{vw} \in \GL(vkQ_1w)$ such that $g_{uv}^{\mu}W(\alpha)g_{vw}^T = W(\alpha)$. For each $\beta$ with $t(\beta) = (u,v)$, we can find another diagonal matrix $g_{s(\beta)}$ satisfying $g_{s(\beta)}W(\alpha)g_{(u,v)}^T = W(\alpha)$. Now repeat this procedure for the targets of the $\alpha$'s and the sources of the $\beta$'s. Since the connected component of $r_2Q$ containing $(u,v)$ is an oriented tree, we can iterate this process a finite number of times to build $g$ on this connected component. We can finish the construction of $g$ by defining it to be the identity on the remaining components. We have shown that $\pi_{uv}(G)$ contains the diagonal matrices. But $\pi_{uv}(G) = \phi_{uv}\pi_{uv}(\operatorname{Vl}_{(Q,I)})\phi_{uv}^{-1}$, so that $\pi_{uv}(\operatorname{Vl}_{(Q,I)})$, and hence $H_B(u,v),$ must contain a maximal torus of $\GL(ukQ_1v)$. Then $ukQ_1v$ is a finite union of $H_B(u,v)$-orbits, as we wished to show.
\end{proof}  

\noindent {\bf{Note:}} For any quiver $Q$, the condition that an $r_2$-ideal $\lambda$ be monomial is equivalent to the condition that each matrix in $V_{\lambda}$ has at most one non-zero entry. But the map  

\begin{center} 
$\displaystyle\begin{array}{c} \{ \text{$r_2$-generating sets of $r_2$-ideals of $kQ$} \} \rightarrow \rep_{\underline{d}}(r_2(Q)) \\ \lambda \mapsto V_{\lambda} \end{array}$
\end{center} 

\noindent is a surjection whenever the no-double-edge graph of $Q$ is a disjoint union of oriented trees. Hence, this allows us to easily construct examples of non-monomial bound quiver algebras $kQ/I$, whose Ext-quiver is not Schur, with the property that $\msa (Q)$ is a finite union of $\Aut_k(kQ/I)$-orbits. 

\begin{example} 
Let $Q$, $I_2$, and $\mu$ be as in example \ref{r2Ex}. Then $V_{\mu}$ satisfies the conditions of proposition \ref{r2FOP}, and so $\msa (Q)$ is a finite union of $\Aut_k(kQ/I_2)$-orbits. Of course, $Q$ is not Schur, and $kQ/I_2$ is not a monomial algebra.
\end{example}

\end{document}